\documentclass[11pt,twoside]{amsart}
\usepackage{amssymb,amsmath,amsthm, amscd, enumerate, mathrsfs, calligra}
\usepackage[all]{xy}
\usepackage{comment}
\usepackage{color}

\title[Zariski density for surfaces]
{Zariski density of points with maximal arithmetic degree for surfaces}
\author[K.~Sano, T.~Shibata]
{Kaoru Sano, Takahiro Shibata}
\date{}

\address{Faculty of Science and Engineering, Doshisha University, Kyoto, 610-0394, Japan.}
\email{kaosano@mail.doshisha.ac.jp}

\address{National University of Singapore, Singapore 119076, Republic of Singapore}
\email{mattash@nus.edu.sg}

\DeclareMathOperator{\sHom}{\mathscr{H}\kern -.3pt \mathit{om}}

\DeclareMathOperator{\Aut}{Aut}

\DeclareMathOperator{\NS}{NS}
\DeclareMathOperator{\N}{N}

\DeclareMathOperator{\Gal}{Gal}

\newcommand{\C}{\mathbb{C}}

\newcommand{\PP}{\mathbb{P}}

\newcommand{\Q}{\mathbb{Q}}

\newcommand{\R}{\mathbb{R}}

\newcommand{\Z}{\mathbb{Z}}

\newenvironment{parts}[0]{%
  \begin{list}{}%
    {\setlength{\itemindent}{0pt}
     \setlength{\labelwidth}{1.5\parindent}
     \setlength{\labelsep}{.5\parindent}
     \setlength{\leftmargin}{2\parindent}
     \setlength{\itemsep}{0pt}
     }%
   }%
  {\end{list}}
\newcommand{\Part}[1]{\item[\upshape#1]}

\newtheorem{thm}{Theorem}[section]
\newtheorem{lem}[thm]{Lemma}

\newtheorem{prop}[thm]{Proposition}
\newtheorem{conj}[thm]{Conjecture}
\newtheorem{ques}[thm]{Question}

\theoremstyle{definition}
\newtheorem{defn}[thm]{Definition}
\newtheorem{rem}[thm]{Remark}

\newtheorem*{ack}{Acknowledgments}

\newtheorem*{notation}{Notation and Conventions}

\begin{document}
\begin{abstract}
We prove that any surjective self-morphism $f$ with $\delta_f > 1$ on a potentially dense smooth projective surface defined over a number field $K$ has densely many $L$-rational points for a finite extension $L/K$.
\end{abstract}

\maketitle

\setcounter{tocdepth}{1}
\tableofcontents
\section{Introduction}\label{sec_intro}
Let $X$ be a projective variety defined over a number field $K$ and
$f \colon X \longrightarrow X$ a surjective self-morphism over $K$.
Then we can define two dynamical quantities,
the dynamical degree $\delta_f$ and the arithmetic degree $\alpha_f(x)$ at a point $x\in X(\overline{K})$,
see Section \ref{sec: def} for the definitions.

Relationships of those two quantities are studied in several papers.
The following result is fundamental.
\begin{thm}[{\cite[endomorphism case]{KS16b}}]
Let $X$ be a projective variety over a number field $K$,
and $f\colon X\longrightarrow X$ a surjective self-morphism over $K$.
Then the limit defining $\alpha_f(x)$ converges and the inequality $\alpha_f(x) \leq \delta_f$ holds for all $x\in X(\overline{K})$.
\end{thm}
See \cite[Theorem 26]{KS16a} and \cite[Theorem 1.4]{Mat16} for the case that $f$ is a dominant rational map.

Matsuzawa, Meng, Zhang, and the second author gave the following conjecture in \cite{MMSZ20}.

\begin{conj}\label{conj: sAND}
	Let $X$ be a projective variety over a number field $K$,
	$f\colon X \longrightarrow X$ a surjective morphism, and
	$d > 0$ a positive integer.
	Then the set
	\[
		Z_f(d) := \{ x \in X(\overline{K}) \ |\ [K(x):K] \leq d, \alpha_f(x) < \delta_f\}
	\]
	is not Zariski dense.
\end{conj}

Roughly speaking, Conjecture \ref{conj: sAND} says that the set of points with non-maximal arithmetic degree is small.

On the other hand, the set of points with maximal arithmetic degree should be large.
To give a precise statement, we prepare the notion of densely many rational points with maximal arithmetic degree.

\begin{defn}
	Let $X$ be a projective variety over a number field $K$ and
	$f\colon X \longrightarrow X$ a surjective self-morphism over $K$.
	Fix an algebraic closure $\overline{K}$ of $K$ and let $L$ be an intermediate field extension $\overline{K} / L / K$.
	We say that \textit{$(X,f)$ has densely many $L$-rational points with maximal arithmetic degree}
	if there is a subset $S\subset X(L)$ satisfying the following conditions:
	\begin{parts}
		\Part{(1)} $S$ is Zariski dense,
		\Part{(2)} the equality $\alpha_f(x) = \delta_f$ holds for all $x \in S$, and
		\Part{(3)} $O_f(x) \cap O_f(y) = \emptyset$ for any distinct two points $x,y \in S$, where $O_f(x)$ is the forward $f$-orbit $\{f^n(x)\ |\ n \geq 0 \}$ of $x$.
	\end{parts}
	If $(X,f)$ has densely many $L$-rational points with maximal arithmetic degree, we also say that \textit{$(X,f)$ satisfies $(DR)_L$}, for abbreviation.
	If there is a finite extension $L/K$ such that $(X,f)$ satisfies $(DR)_L$,
	we say that \textit{$(X,f)$ satisfies $(DR)$}.
\end{defn}

In \cite{SS20}, the authors proved that
for a surjective self-morphism $f\colon X \longrightarrow X$ on a projective variety $X$ both defined over a number field $K$,
$(X,f)$ has densely many $\overline{K}$-rational points with maximal arithmetic degree.
The authors also gave the following question:

\begin{ques}\label{ques: potential}
Let $X$ be a projective variety and $f \colon X \longrightarrow X$ a surjective self-morphism.
If $X$ has potentially dense rational points i.e.~there is a finite extension $L/K$
such that $X(L)$ is Zariski dense,
then does $(X,f)$ satisfy $(DR)$?
\end{ques}

In fact, if $X$ is either a unirational variety, an abelian variety, or a $\PP^1$-bundle over an elliptic curve,
the authors gave affirmative answers for Question \ref{ques: potential} in \cite{SS20}.
The main result of this paper is that Question \ref{ques: potential} is affirmative if $X$ is a smooth projective surface and $f\colon X\longrightarrow X$ is a surjective self-morphism with $\delta_f > 1$.

\begin{thm}\label{thm: main}
Let $K$ be a number field, $X$ a smooth projective surface over $K$ having potentially dense rational points,
and $f\colon X \longrightarrow X$ a surjective self-morphism with $\delta_f>1$.
Then $(X,f)$ satisfies $(DR)$.
\end{thm}

The idea of the proof is as follows.
Replacing the self-morphism by its iteration, a self-morphism on the minimal model is induced.
So we may assume that the given surface is minimal.
Since the case of abelian varieties and unirational varieties is proved in \cite{SS20},
the remaining case is automorphisms of K$3$ surfaces and non-isomorphic surjective self-morphisms of elliptic surfaces.
These cases are treated in case-by-case analysis.

The outline of this paper is as follows.
In Section \ref{sec: def}, we prepare some notation and definitions which we use in this paper.
In Section \ref{sec: lem}, lemmata to be used in the proof of Theorem \ref{thm: main} are prepared.
We prove Theorem \ref{thm: main} for automorphisms of K$3$ surfaces and non-isomorphic surjective self-morphisms of elliptic surfaces
in Section \ref{subsec: lem for autom} and Section \ref{subsec: non-isom}, respectively.

\begin{ack}
The authors thank Professor Shu Kawaguchi for giving valuable comments and suggesting them writing this paper, and Professor Joseph Silverman for reading a draft and giving valuable comments.
The first author is supported by JSPS KAKENHI Grant Number JP20K14300.
The second author is supported by a Research Fellowship of NUS.
This work was supported by the Research Institute for Mathematical Sciences,
an International Joint Usage/Research Center located in Kyoto University.
\end{ack}

%%%%%%%%%%%%%%%%%%%%%%%%%%%%
%定義とノーテーション
%%%%%%%%%%%%%%%%%%%%%%%%%%%%
\section{Definitions and Notation}\label{sec: def}

\begin{notation}
\begin{itemize}
$ \, $
\item Throughout this article, we work over a fixed number field $K$.
We fix an algebraic closure $\overline K$ of $K$.

\item A \textit{variety} means a geometrically integral separated scheme of finite type over $K$.

\item The symbols $\sim$ (resp.~$\sim_{\mathbb Q}$, $\sim_{\mathbb R}$) and
$\equiv$  mean
the linear equivalence (resp.~$\mathbb Q$-linear equivalence,
$\mathbb R$-linear equivalence) and the numerical equivalence on divisors.

\item Let $\mathbb K= \mathbb Q, \mathbb R$ or $\mathbb C$.
For a $\mathbb K$-linear endomorphism $\phi: V \to V$ on a $\mathbb K$-vector space $V$,
$\rho(\phi)$ denotes the {\it spectral radius} of $f$, that is,
the maximum of absolute values of eigenvalues (in $\C$) of $\phi$.

\item Though the definition of the dynamical degree of dominant rational self-map is known,
we need them only for surjective self-morphisms in this paper.
Let $f \colon X \dashrightarrow X$ be a dominant rational map on a projective variety.
We define the (\textit{first}) \textit{dynamical degree $\delta_f$ of $f$} as
\[
\delta_f=\lim_{n \to \infty} ((f^n)^*H \cdot H^{\dim X-1})^{1/n}.
\]

Let $f^*:\NS(X)_\R \to \NS(X)_\R$ be the pull-back action on the space of
numerical classes of $\R$-Cartier divisors on $X$.
If $f$ is a morphism, then $\delta_f=\rho(f^*)$, 
so $\delta_f$ is an algebraic integer.

\item Let $X$ be a projective variety.
For an $\mathbb R$-Cartier divisor $D$ on $X$,
a function $h_D: X(\overline{K}) \to \mathbb R$ is determined up to the
difference of a bounded function.
$h_D$ is called the \textit{height function associated to $D$}.
For definition and properties of height functions, see e.g.~\cite[Part B]{HS00} or \cite[Chapter 3]{Lan83}.

\item Let $X$ be a projective variety and $f: X \longrightarrow X$ a surjective self-morphism.
Fix an ample height function $h_H \geq 1$ on $X$.

\item
For $x \in X(\overline K)$, we define
\[
\alpha_f(x)=\limsup_{n \to \infty} h_H(f^n(x))^{1/n},
\]
which we call the \textit{arithmetic degree of $f$ at $x$}.
The convergence defining the arithmetic degree is known
 (cf.~\cite{KS16b}).
Moreover, the arithmetic degree is independent of the choice of $H$ and $h_H$.
\end{itemize}
\end{notation}

\begin{rem}
Dynamical degrees have the following invariance: 
if $\pi \colon X \dashrightarrow X'$
is a dominant rational map between projective varieties of the same dimension,
and $f \colon X\dashrightarrow X$ and $f'\colon X' \dashrightarrow X'$ are 
dominant rational maps satisfying $\pi \circ f = f' \circ \pi$, 
then the equality $\delta_f = \delta_{f'}$ holds.
For details on dynamical degrees, see \cite{Tru20}.
\end{rem}

%%%%%%%%%%%%%%%%%%%%%%%%%%%%
%補題達
%%%%%%%%%%%%%%%%%%%%%%%%%%%%
\section{Lemmata} \label{sec: lem}
In this section, we list lemmata used in the next section.

%%%%%%%%%%%%%%%%%%%%%%%%%%%%
%\subsection{Reductions} \label{subsec: lem for red}
%%%%%%%%%%%%%%%%%%%%%%%%%%%%
\begin{lem}\label{lem: invariance} %generically finite invariance of alpha, delta
	Consider the following commutative diagram
	\[
	\xymatrix{
	X \ar[r]^{f_X} \ar[d]_{\pi} & X \ar[d]^{\pi}\\
	Y\ar[r]_{f_Y} &Y,
	}
	\]
	where $X,Y$ are smooth projective varieties and $f_X$, $f_Y$ are surjective self-morphisms.
	Suppose that there exists a non-empty open subset $U\subset Y$ such that
	$\pi \colon V:= \pi^{-1}(U) \longrightarrow U$ is finite.
	Let $x\in X(\overline{K})$, $y:= \pi(x) \in Y(\overline{K})$, $O_{f_X}(x) \subset V$, $O_{f_Y}(y) \subset U$.
	Then the equality $\alpha_{f_X}(x) = \alpha_{f_Y}(y)$ holds.
\end{lem}
\begin{proof}
	Since $f_X$ and $f_Y$ are morphism, the existence of the limit defining the arithmetic degrees are known.
	The equality $\alpha_{f_X}(x) = \alpha_{f_Y}(y)$ is a part of \cite[Lemma 2.8]{MS20}.
\end{proof}

\begin{lem}\label{lem: genefin reduction}%finite
	Let $X,Y$ be projective varieties,
	 $f_X\colon X\longrightarrow X$ and $f_Y\colon Y \longrightarrow$ 
	surjective self-morphisms on $X$ and $Y$, respectively,
	and $\pi \colon X \longrightarrow Y$ a surjective morphism such that $\pi \circ f_X = f_Y \circ \pi$.
	\begin{parts}
		\Part{(a)} If $\pi$ is birational and $(Y, f_Y)$ satisfies $(DR)$,
			then $(X, f_X)$  satisfies $(DR)$.

		\Part{(b)} If $\pi$ is finite and $(X,f_X)$  satisfies the  $(DR)$,
			then $(Y, f_Y)$ satisfies  $(DR)$.
	\end{parts}
\end{lem}

\begin{proof}
	\begin{parts}
		\Part{(a)} Let $L$ be a finite extension of $K$ such that $(Y, f_Y)$ satisfies $(DR)_L$.
			Let $S_Y = \{ y_i\}_{i=1}^{\infty} \subset Y(L)$ be a sequence of $L$-rational points such that
			\begin{itemize}
				\item $S_Y$ is Zariski dense,
				\item $\alpha_{f_Y}( y_i) = \delta_{f_Y}$ for $i \geq 1$, and
				\item $O_{f_Y} (y_i) \cap O_{f_Y}(y_j) = \emptyset$ for $i\neq j$.
			\end{itemize}
			Let $\{X_j \}_{j=1}^{\infty}$ be the family of all the proper closed subsets of $X$.
			Let $U\subset X$ and $V\subset Y$ be open subsets such that $\pi|_U \colon U\longrightarrow V$ is isomorphic.
			Then we can inductively take a subsequence $\{y_{i_j}\}_{j=1}^{\infty} \subset V$ of $S_Y$ such that
			$x_j := (\pi|_{U})^{-1}(y_{i_j}) \not \in X_j$ for $j\geq 1$.
			Since $\alpha_{f_X} (x_j) =\alpha_{f_Y}(y_{i_j})$ and $\delta_{f_X} = \delta_{f_Y}$ by Lemma \ref{lem: invariance},
			the sequence $S_X := \{x_j\}_{j=1}^{\infty}$ satisfies
			\begin{itemize}
				\item $x_j \not \in X_j$ for $j\geq 1$,
				\item $\alpha_{f_X}(x_j) = \delta_{f_X}$ for $j \geq 1$, and
				\item $O_{f_X}(x_j) \cap O_{f_X}(x_k) = \emptyset$ for $j\neq k$.
			\end{itemize}
			
		\Part{(b)} Let $L$ be a finite extension of $K$ such that $(X, f_X)$ satisfies $(DR)_L$.
			Let $S_X = \{ x_i\}_{i=1}^{\infty} \subset X(L)$ be a sequence of $L$-rational points such that
			\begin{itemize}
				\item $S_X$ is Zariski dense,
				\item $\alpha_{f_X}( x_i) = \delta_{f_X}$ for $i \geq 1$, and
				\item $O_{f_X} (x_i) \cap O_{f_X}(x_j) = \emptyset$ for $i\neq j$.
			\end{itemize}
			Let $\{Y_j \}_{j=1}^{\infty}$ be the family  of all the proper closed subsets of $Y$.
			Then since $\pi$ is finite, the number of $k \in \Z_{\ge 1}$ such that $O_{f_Y}(\pi(x_k)) \cap O_{f_Y}(\pi(x_i)) \neq \emptyset$ is finite for each $i \ge 1$.
			So we can inductively take a subset $\{x_{i_j}\}_{j=1}^{\infty}$
			such that
			\begin{itemize}
				\item $y_j := \pi(x_{i_j}) \not \in Y_j$ for $j\geq 1$,
				\item $\alpha_{f_Y}(y_j) = \alpha_{f_X} (x_{i_j}) = \delta_{f_X} = \delta_{f_Y}$, and
				\item $O_{f_Y}(y_j) \cap O_{f_Y}(y_k) = \emptyset$ for $j\neq k$,
			\end{itemize}
			where the second assertion follows by Lemma \ref{lem: invariance}.
	\end{parts}
\end{proof}

\begin{lem} \label{lem: iterate reduction}
	Let $X$ be a projective variety and 
	$f\colon X\longrightarrow X$ a surjective self-morphism.
	Then for any integer $t \geq 1$ and a finite extension $L/K$,
	$(X,f)$ satisfies $(DR)_L$
	if and only if $(X, f^t)$ satisfies $(DR)_L$.
\end{lem}
\begin{proof}
	Obviously $(X,f^t)$ satisfies $(DR)_L$ if so does $(X,f)$.
	Conversely, assume that $(X,f^t)$ satisfies $(DR)_L$.
	Let $S = \{ x_i \}_{i=1}^{\infty} \subset X(L)$ be a subset such that
	\begin{itemize}
		\item $S$ is Zariski dense,
		\item $\alpha_{f^t}(x_i) = \delta_{f^t}$ for $i = 1,2, \ldots$,
		\item $O_{f^t}(x_i) \cap O_{f^t}(x_j) = \emptyset$ for $i\neq j$.
	\end{itemize}
	Note that we have $\alpha_f(x) =\alpha_{f^t}(x)^{1/t} = \delta_{f^t}^{1/t} =  \delta_f$ for $x\in S$.
	Thus, it is enough to prove that
	for any proper closed subset $Y\subset X$ and any points $x_1',x_2', \ldots x_r' \in S$,
	there is a point $x \in S$ such that $x \not \in Y$ and $O_f(x) \cap O_f(x_i') = \emptyset$ for $1 \le i \le r$.
	Since $S$ is Zariski dense, there are infinitely many $i \geq 1$ such that $x_i \not \in Y$.
	Since the number of $i \ge 1$ such that $O_f(x_i) \cap O_f(x_j') \neq \emptyset$ is at most $t$ for each $j$.
	Hence we can get a point $x_i \in S$ which we wanted.
\end{proof}

%%%%%%%%%%%%%%%%%%%%%%%%%%%%
%非自己同型に関する補題の節
%%%%%%%%%%%%%%%%%%%%%%%%%%%%
%\subsection{On non-isomorphic surjective self-morphisms}

\begin{thm}\label{thm: uniform bound}
Let $K$ be a number field.
Then there is a constant $N$ such that, for any elliptic curve defined over $K$,
the number of $K$-torsion points is at most $N$.
\end{thm}
\begin{proof}
See \cite{Mer96}.
\end{proof}

\begin{lem}\label{lem: deinsity of ellipric surface}
	Let $\pi \colon X \longrightarrow B$ be an elliptic surface.
	Then the followings are equivalent.
	\begin{itemize}
		\item[(a)] $X(K)$ is Zariski dense.
		\item[(b)] There are infinitely many points $b\in B(K)$ such that $X_b = \pi^{-1}(b)$ has infinitely many $K$-rational points.
	\end{itemize}
\end{lem}
\begin{proof}
Clearly (b) implies (a).
Assume that (b) does not hold.
Then we can take an open subset $B' \subset B$ such that 
$X' = \pi^{-1}(B') \overset{\pi}{\longrightarrow} B'$ admits the structure of 
an abelian scheme and $X_b(K)$ is finite for any $b \in B'(K)$.
Let $N$ be an upper bound of the number of torsion points of elliptic curves defined over $K$, which we can take by Lemma \ref{thm: uniform bound}.
Let $[N!] \colon X' \longrightarrow X'$ be the morphism defined by the multiplication 
map by $N!$ on each fiber.
Then we have
\[
X(K) \subset \ker ([N!])\ \cup\ (\pi^{-1}(B \setminus B'))(K).
\]
Hence $X(K)$ is not dense.
\end{proof}

\begin{lem}\label{lem: polarized}
Let $f\colon X\longrightarrow X$ be a surjective self-morphism with $\delta_f >1$
on a normal projective variety.
Assume that there is an ample $\R$-Cartier divisor $H$ such that $f^\ast H \equiv_{\R} \delta_f H$.
\begin{parts}
	\Part{(a)} The number of the preperiodic points are finite.
	\Part{(b)} $x\in X(\overline{K})$ is not preperiodic if and only if $\alpha_f(X) = \delta_f$.
\end{parts}
\end{lem}
\begin{proof}
By \cite[Theorem 6.4 (1)]{MMSZZ20}, there is an ample $\R$-Cartier divisor $D'$ such that $f^\ast D' \sim_{\R} \delta_f D'$.
So the assertion follows from \cite[Corollary 1.1.1]{CS93}.
\end{proof}

%%%%%%%%%%%%%%%%%%%%%%%%%%%%
%自己同型に関する節
%%%%%%%%%%%%%%%%%%%%%%%%%%%%
\section{Ptoof of Theorem \ref{thm: main}}
We divide the proof of Theorem \ref{thm: main} into two cases: the automorphism case and the non-isomorphic self-morphism case.

\subsection{The automorphism case} \label{subsec: lem for autom}
We start with listing known results on automorphisms of surfaces.
\begin{lem}\label{lem: eigen}
	Let $X$ be a smooth projective surface over $\C$, and $f\colon X \longrightarrow X$ be an automorphism with $\delta_f > 1$.
	\begin{parts}
		\Part{(a)} The set of eigenvalues of $f^\ast |_{H^{2}(X,\R)}$ with counted multiplicities is
			\[
			\{ \delta_f, \delta_f^{-1}, \lambda_1, \lambda_2, \ldots \lambda_{\dim H^2(X, \R) -2}\},
			\]
			where $|\lambda_i| = 1$ for all $i= 1,2,\ldots , \dim H^2(X,\R)-2$.

		\Part{(b)} The eigenvalues $\delta_f$ and $\delta_f^{-1}$ are irrational real numbers.
			Moreover, $\delta_f^{-1}$ is a Galois conjugate of  $\delta_f$ over $\Q$.

		\Part{(c)} There are numerically non-zero nef $\R$-divisors $D^+$ and $D^-$
			satisfying $f^\ast D^+ \sim_{\R} \delta_f D^+$ and $f^\ast D^- \sim_{\R} \delta_f^{-1} D^-$, respectively.

		\Part{(d)} For a curve $C$ in $X$, $(C\cdot D^+) = 0$ holds if and only if $(C\cdot D^{-}) = 0$.
		
		\Part{(e)} Let $D := D^+ + D^-$. Then the set $\mathcal{C}_f$ of irreducible curves $C$
				satisfying $(C\cdot D) = 0$ is a finite set.
		
		\Part{(f)} The set $\mathcal{C}_f$ coincides with the set of $f$-periodic irreducible curves in $X$.
		
		\Part{(g)} For $\bullet \in \{ +, - \}$, the set $\mathcal{C}_f$ coincides
				with the set of irreducible curves $C$ such that $(C\cdot D^{\bullet}) = 0$ holds.
	\end{parts}
\end{lem}

\begin{proof}
	\begin{parts}
		\Part{(a)} See \cite[Theorem 3.2]{McM02} and \cite[Theorem 2.1]{Kaw08}.
		
		\Part{(b)} Since $f^\ast |_{H^2(X, \R)}$ is induced by the action of $f^\ast$ on $H^2(X,\Z)$,
				an integer matrix represents $f^\ast |_{H^2(X, \R)}$.
				So $\delta_f$ and $\delta_f^{-1}$ are algebraic integers.
				If $\delta_f$ is a rational number, $\delta_f^{-1}$ is also a rational number,
				so $\delta_f$ and $\delta_f^{-1}$ are integers.
				Hence we get $\delta_f = \delta_f^{-1} = 1$.
				This is contradiction.
				Since $f$ is an isomorphism,
				the constant term of the characteristic polynomial of $f^\ast |_{H^2(X,\R)}$ is $1$.
				Since the minimal polynomial of $\delta_f$ have integer coefficients and
				divides the characteristic polynomial of $f^\ast |_{H^2(X,\R)}$,
				the constant term of the minimal polynomial of $\delta_f$ is also $1$.
				Since $|\lambda_i| = 1$ for $1\leq i \leq \dim H^2(X,\R) -2$,
				the number $\delta_f^{-1}$ must be a Galois conjugate of $\delta_f$ over $\Q$.

		\Part{(c)} See \cite[Proposition 2.5]{Kaw08}.

		\Part{(d)} Let $F$ be the Galois closure of $\Q(\delta_f) / \Q$.
				Let $\sigma \in \Gal(F/ \Q)$ be an automorphism sending $\delta_f$ to $\delta_f^{-1}$.
				Then since $D^+$ and $D^-$ lie in nef classes in $\N^1(X) \otimes_{\Z} \Q(\delta_f)$,
				we have $(C \cdot D^{\bullet}) \in \Q(\delta_f) (\subset F)$ for any curve $C$ in $X$ and for $\bullet \in \{ +, -\}$.
				Since $\delta_f$ and $\delta_f^{-1}$ are Galois conjugate over $\Q$,
				we have $\sigma D^+ = \sigma D^-$, so $\sigma(C\cdot D^{+}) = (C\cdot D^{-})$.

		\Part{(e), (f)} See \cite[Proposition 3.1]{Kaw08}

		\Part{(g)} If $(C \cdot D) = 0$ holds, we have $(C\cdot D^\bullet ) = 0$ for $\bullet \in \{ +, - \}$ since $D^+$ and $D^-$ are nef.
				If $(C\cdot D^+) = 0$ holds, we get $(C\cdot D^-) = 0$ by (d), so $(C\cdot D) = (C\cdot D^+) + (C\cdot D^-) = 0$.
				If $(C\cdot D^-) = 0$ holds, the similar is true.
	\end{parts}
\end{proof}

\begin{thm} \label{thm: classification of auto}
	If $X$ is a smooth projective surface admitting an automorphism $f$ with $\delta_f > 1$,
	then $X$ is either a non-minimal rational surface, or a surface birational to a K$3$ surface, an Enriques surface, or an abelian surface.
\end{thm}
\begin{proof}
See \cite[Proposition 1]{Can99}.
\end{proof}

\begin{prop}\label{prop: BT}
	Let $X$ be a K$3$ surface over a number field $K$ with an infinite group of automorphisms.
	Then there is a rational curve $C \subset X$ such that $\#\{g(C)\ |\ g \in \Aut(X)\} = \infty$.
\end{prop}
\begin{proof}
	See the proof of \cite[Theorem 4.10]{BT00}.
\end{proof}

\begin{prop}\label{prop: rational curve}
	Let $X$ be a K$3$ surface defined over a number field $K$,
	and $f\colon X\longrightarrow X$ an automorphism with $\delta_f > 1$.
	Then there exists a rational curve $C \subset X$ such that $\#\{ f^n(C) \ |\ n \geq 0\} = \infty$
\end{prop}
\begin{proof}
	Since $X$ contains only finitely many $f$-periodic curves by Lemma \ref{lem: eigen}{(e),(f)},
	and there are infinitely many rational curves in $X$ by Proposition \ref{prop: BT},
	there is a rational curve $C\subset X$ which is not $f$-periodic.
\end{proof}

\begin{prop}\label{prop: dag}
	Let $X$ be a projective variety and $f: X \longrightarrow X$ a surjective morphism with $\delta_f>1$.
	Assume the following condition:

		\vspace{0.15in}
		\rm (\dag): \it There is a nef $\R$-Cartier divisor $D$ on $X$ such that
		$f^*D \sim_\R \delta_fD$ and
		for any proper closed subset $Y \subset X_{\overline K}$,
		there exists a non-constant morphism $g: \PP^1_K \longrightarrow X$ such that $g(\PP^1_K) \not\subset Y$
		and $g^*D$ is ample.
		\vspace{0.15in}

	Then $(X,f)$ satisfies the condition $(DR)_K$.
\end{prop}
\begin{proof}
See \cite[Theorem 4.1]{SS20}.
\end{proof}

\begin{proof}[Proof of Theorem \ref{thm: main} when $f$ is an automorphism]
	If $X$ is rational, our assertion is true by \cite{SS20}.
	Assume that $X$ is irrational.
	Take a birational morphism $\mu \colon X \longrightarrow Y$ to
	the minimal model $Y$ of $X$ and
	let $g : Y \dashrightarrow Y$ be the birational automorphism on $Y$
	defined as $g = \mu \circ f \circ \mu^{-1}$.
	Then $g$ is in fact an automorphism
	since, if $g$ has indeterminacy, $Y$ must have a $K_Y$-negative curve.
	By Theorem \ref{thm: classification of auto} and Lemma \ref{lem: genefin reduction},
	we may assume that $X$ is either a K$3$ surface, an Enriques surface, or an abelian variety.
	If $X$ is an abelian variety, our assertion is true by \cite{SS20}.
	If $X$ is an Enriques surface, take the universal covering $\pi \colon \tilde{X} \longrightarrow X$.
	Then an automorphism $\tilde{f} \colon \tilde{X} \longrightarrow \tilde{X}$ such that
	$\pi \circ \tilde{f} = f\circ \pi$ is induced and $\tilde{X}$ is a K$3$ surface.
	Hence by Lemma \ref{lem: genefin reduction}, we may assume that $X$ is a K$3$ surface.
	
	Now it is enough to prove that \rm(\dag) in Proposition \ref{prop: dag}.
	Let $Y$ be a proper closed subset of $X$.
	Take a rational curve $\iota \colon C \hookrightarrow X$
	suth that $\# \{ f^n \circ \iota (C)\ |\ n \geq 0\} = \infty$ by Proposition \ref{prop: rational curve}.
	Then there is an integer $n_Y \geq 0$ such that $C_Y := f^{n_Y} \circ \iota (C) \not \subset Y$.
	Let $g_Y := f^{n_Y} \circ \iota$.
	Since $C_Y$ is not $f$-periodic, we have $(C_Y\cdot D^+) > 0$ by Lemma \ref{lem: eigen} (e), (f),
	so $g_Y^{\ast} D^+$ is ample.
	Hence \rm(\dag) in Proposition \ref{prop: dag} is proved.
\end{proof}

%%%%%%%%%%%%%%%%%%%%%%%%%%%%
%同型でない自己全射に関する主節
%%%%%%%%%%%%%%%%%%%%%%%%%%%%

\subsection{The non-isomorphic surjective self-morphism case}\label{subsec: non-isom}
We prepare the following lemmata to reduce to a minimal surface.

\begin{lem}\label{lem: kappa-inf}
Let $f\colon X\longrightarrow X$ be a non-isomorphic surjective self-morphism
on a smooth projective irrational surface $X$ with $\kappa(X) = -\infty$.
Then there is a positive integer $t$,
a birational morphism $\mu \colon X \longrightarrow X'$ to a $\PP^1$-bundle over a curve $B$ with $g(B)\geq 1$,
and a surjective self-morphism $f'\colon X'\longrightarrow X'$
such that the equality $\mu \circ f^t = f' \circ \mu$ holds.
\end{lem}
\begin{proof}
By \cite[Proposition 10]{Nak02}, any $(-1)$-curve on $X$ is $f$-periodic.
So the assertion follows.
\end{proof}

\begin{lem}\label{lem: kappa0}
Let $f\colon X \longrightarrow X$ be a non-isomorphic surjective self-morphism on a smooth projective surface $X$ with $\kappa(X) \geq 0$.
Then $X$ is minimal.
\end{lem}
\begin{proof}
See  \cite[Lemma 2.3 and Proposition 3.1]{Fuj02}.
\end{proof}

\begin{proof}[Proof of Theorem \ref{thm: main} when $f$ is not an automorphism]
We prove it by using  the Enriques--Kodaira Classification and case-by-case analysis.
\begin{parts}
\Part{(I)} $\kappa(X) = - \infty$.
		By Lemma \ref{lem: kappa-inf}, Lemma \ref{lem: genefin reduction}, and Lemma \ref{lem: iterate reduction},
		we may assume that $X$ is either a rational surface or a $\PP^1$-bundle over a curve $B$ with genus $g(B) \ge 1$.
		If $X$ is rational, the assertion follows by \cite[Theorem 1.11]{SS20}.
		If $g(B) \geq 2$, $X$ does not have potentially dense rational points.
		If $g(B) = 1$, $X$ is a $\PP^1$-bundle over an elliptic curve.
		This case is proved in \cite[Theorem 6.1]{SS20}.
\end{parts}
	
From now on, assume that $\kappa(X) \geq 0$. Then $X$ is minimal by Lemma \ref{lem: kappa0}.

\begin{parts}
\Part{(II)} $\kappa(X) = 0$.
By \cite[Theorem 3.2]{Fuj02}, $X$ is either a hyperelliptic surface or an abelian surface.

		\begin{itemize}
			\item[(II-i)] The hyperelliptic surface case.
					In this case, there is an equivariant finite covering from an abelian variety.
					See e.g.~the proof of \cite[Theorem 7.1]{MSS18} for details.
					By taking such an equivariant finite covering and applying Lemma \ref{lem: genefin reduction},
					we can reduce to the abelian surface case.
					
			\item[(II-ii)] The abelian surface case. More generally the abelian 
			variety case is proved in \cite[Theorem 1.12]{SS20}.
		\end{itemize}

\Part{(III)} $\kappa(X) = 1$. We treat this case below.

\Part{(IV)} $\kappa(X) = 2$. This case does not occur since any surjective self-morphisms on $X$ are automorphisms of finite order.
\end{parts}

Thus the remaining case is the $\kappa(X)=1$ case.
Then $X$ admits an elliptic fibration $\pi \colon X\longrightarrow B$ and 
$f$ descends to an automorphism of finite order on $B$ 
(cf.~\cite[Theorem A]{MZ19} or \cite[Theorem 8.1]{MSS18}).
Replacing $f$ by some iteration $f^t$,
we may assume that $f$ is a morphism over the base curve $B$.
Let $\{ Y_i \}_{i = 1}^{\infty}$ be the set of all the proper closed subsets of $X$.
It is enough to get $\{x_i\}_{i=1}^{\infty} \subset X(L)$ such that
\begin{itemize}
	\item $x_i \not \in Y_i$,
	\item $\alpha_f(x_i) = \delta_f$ for $i = 1,2,\ldots$, and
	\item $O_f(x_i) \cap O_f(x_j) = \emptyset$ for $i\neq j$.
\end{itemize}
Let $L$ be a finite extension field of $K$ such that $X(L)$ is Zariski dense.
By Lemma \ref{lem: deinsity of ellipric surface},
we can find an infinite subset $\{b_j \}_{j=1}^{\infty} \subset B(L)$ such that $X_{b_j} = \pi^{-1} (b_j)$
contains infinitely many $L$-rational points for each $j$.
Removing special fibers, we may assume that $X_{b_j}$ is an elliptic curve and 
$f|_{X_{b_j}}$ has dynamical degree $\delta_f$.
By Lemma \ref{lem: polarized}, there are infinitely many $L$-rational points $x\in X_{b_j}(L)$ with $\alpha_f(x) = \delta_f$ for each $j$.
Hence, letting $\{b_{j_i}\}_{i=1}^{\infty}$ be a subsequence of $\{b_j \}_{j=1}^{\infty}$ such that $X_{b_{j_i}}$ is not contained $Y_i$,
we can find $x_i \in X_{b_{j_i}}(L)$ such that $x_i \not \in Y_i$.
Since $f$ is a morphism over $B$, we have $O_f(x_k) \cap O_f(x_l) = \emptyset$ for $k \neq l$.
The assertion is proved.
\end{proof}

\end{document}